\documentclass[11 pt, oneside]{amsart}
\usepackage[utf8]{inputenc}
\usepackage[english]{babel}
\usepackage[dvipsnames]{xcolor}
\usepackage{amsmath}
\usepackage{amssymb}
\usepackage{mathtools}
\usepackage{bm}
\usepackage{tikz}
\usepackage{latexsym,times,color,gensymb}
\usepackage{units}
\usetikzlibrary{automata,arrows,positioning,calc,shapes.misc,fit}
\newcommand{\lbl}{\footnotesize\textcolor{black!50}}
\usepackage[colorinlistoftodos]{todonotes}
\usepackage{booktabs,tabularx}
\usepackage{multirow}
\usepackage{multicol}
\usepackage{blkarray}
\usepackage{paralist}
\usepackage{float}
\usepackage[normalem]{ulem}
\usepackage{graphicx}

\usepackage{geometry}

\newtheorem{theorem}{Theorem}[section]
\newtheorem{corollary}[theorem]{Corollary}

\newtheorem{prop}[theorem]{Proposition}
\newtheorem{conj}[theorem]{Conjecture}

\theoremstyle{definition}
\newtheorem{definition}[theorem]{Definition}

\newtheorem{rmk}[theorem]{Remark}

\newcommand{\Mr}{M^{\mathrm{r}}}

\usepackage{hyperref}

\begin{document}
\title{A Random Card Shuffling Process}

\author{Joel Brewster Lewis and Mehr Rai}
\date{\today}
\maketitle

\begin{abstract}
Consider a randomly shuffled deck of $2n$ cards with $n$ red cards and $n$ black cards. We study the average number of moves it takes to go from a randomly shuffled deck to a deck that alternates in color by performing the following move:  If the top card and the bottom card of the deck differ in color place the top card at the bottom of the deck, otherwise, insert the top card randomly in the deck. We use tools from combinatorics, probability, and linear algebra to model this process as a finite Markov chain.
\end{abstract}

\section*{Introduction}

Fix a positive integer $n$. In this paper, we consider a random shuffling game played on a deck of $2n$ cards, with $n$ cards of one color and $n$ cards of another. This game is inspired by the top-to-random shuffle (as in, e.g., \cite{AldousDiaconis, BHR}), in which the top card is repeatedly randomly inserted into the deck.  Despite its simple description, the resulting Markov chain has a rich structure, with connections to the representation theory of the symmetric group \cite{diaconis_fill_pitman_1992}.  The process we study is defined by the following move: if the cards on the top and at the bottom of the deck are of the same color, we insert the top card randomly into the $2n$ positions in the deck; but if they are of different colors, we deterministically move the top card to the bottom of the deck.

Unlike the usual top-to-random shuffle, the resulting Markov chain is not recurrent; indeed, we show (in Corollary~\ref{prop:absorbing}) that it is an absorbing chain, with any initial deck eventually converging to a deck in which the two colors alternate, that is, a deck that has cards of one color in all the even positions or all the odd positions of the deck.  Moreover, we show (in Theorem~\ref{prop:tiers}) that there is a \emph{tier} structure on the decks, such that the random process proceeds successively through the various tiers $k$, $k - 1$, \ldots, before reaching the alternating deck (tier $0$).

Our main goal is to study the average number of moves it takes to reach such an alternating deck of cards.  In tier $1$ (the decks closest to the alternating case), we are able to explicitly solve for these expected values (Corollary~\ref{cor:t1steps}).  However, in higher tiers, finding explicit answers seems intractable (see discussion in Section~\ref{sec:exact values}).  Instead, we focus on bounds.  As our main result, we show (in Theorem~\ref{maxmink} and Corollary~\ref{cor:betterMbound}) that the decks in tier $k$ all require in expectation between $n^2 + 2nH_k$ and $(2 H_k - 1)n^2 -nk + O(n)$ moves  (where $H_k$ is the $k$-th harmonic number and $O$ is big O notation) to reach the absorbing deck.  We conjecture (Conjecture~\ref{conj:limit}) that the lower bound is closer to the truth: specifically, we believe that all decks have expected times to absorption of order $n^2$.

We end this introduction with an overview of the paper.  In Section~\ref{2}, we provide necessary terminology and background related to Markov chains in general and our process in particular. We also include an algorithm from \cite{FM:1} that computes the average number of moves it takes to go from any deck to the perfectly alternating deck. Since our goal is to find the average number of moves it takes to reach the absorbing state for an arbitrary $n$, not just for fixed values of $n$, we are motivated to find a general formula. In Section~\ref{3}, we introduce tiers (Definition~\ref{def:tier}), and use them to show that the chain really is absorbing.  In Section~\ref{sec:tier1}, we provide a complete analysis of the expectations for decks in tier $1$ (Corollary~\ref{cor:t1steps}).  In Section~\ref{sec:tier2+}, we provide upper and lower bounds for the expected number of steps to absorption for decks in arbitrary tiers (Theorem~\ref{maxmink}), and in Section~\ref{sec:improved max}, we improve the upper bound (Theorem~\ref{thm:maxMr}).  We end in Section~\ref{4} with some final remarks and open questions.

\subsection*{Acknowledgments}
This paper is based on the second author's undergraduate honors thesis at George Washington University \cite{thesis}.  We thank Joseph Bonin and Daniel Ullman for comments and suggestions.  Work of the first author was supported in part by a Simons Foundation grant (634530) and the GW University Facilitating Fund.

\section{Background and preliminaries}
\label{2}

In this section, we introduce some notations and conventions that will be used throughout the paper.  For definitions related to Markov chains, we follow \cite{FM:1}.

Ignoring suits and numbers, a standard deck of cards consists of an equal number of cards in each of two colors, which we will call $0$ and $1$.
\begin{definition} 
A \textbf{deck} is an ordered sequence consisting of $n$ copies of $0$ and $n$ copies of $1$. We represent the deck as a string of $0$s and $1$s (omitting commas and parentheses).
\label{def:deck}
\end{definition}
Thus, there are $\binom{2n}{n}$ decks consisting of $n$ cards in each of the two colors.

Next we will consider a random dynamical process that moves from one deck to another.  This process will be a \textbf{Markov chain}: that is, the probability of moving to a given state is dependent only on the present state.  The moves in the Markov chain are defined as follows.

\begin{definition} 
Given a deck $D$ whose first and last digits are different, let $D'$ be the deck that has $D$'s first digit as its last digit and the rest of the digits of the deck are shifted to the left by one position.  Then a \textbf{move} sends $D$ to $D'$ with probability $1$.  In this case we call $D$ a \textbf{deterministic deck} or \textbf{d-deck}.

Given a deck $D$ whose first and last digits are identical, then a move sends $D$ to a deck $D'$ by inserting the first digit of $D$ randomly in any of its $2n$ positions with all choices being equally likely. In this case we call $D$ a \textbf{random deck} or an \textbf{r-deck}. 
\label{def: move}
\end{definition}
These moves are illustrated in Figure~\ref{fig:moveexample} for a deck of length $6$.


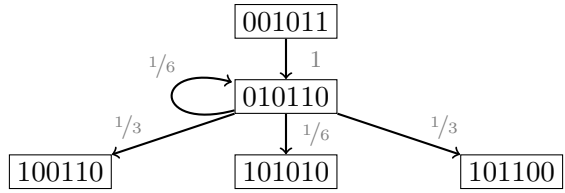
\begin{figure}
\begin{center}
\begin{tikzpicture}  
   \tikzset{Vertex/.style = {rectangle,inner sep=1mm, draw=black}}  

  \node[Vertex] (1) at (0,0)  {$001011$}; 
  \node[Vertex] (2) at (0,-1) {$010110$};
\foreach \from/\to in
{1/2}
\draw[->,thick](\from)--(\to);
    \node[Vertex] (3) at (-3,-2) {$100110$};
    \foreach \from/\to in
{2/3}
\draw[->,thick](\from)--(\to);
      \node[Vertex] (4) at (0,-2) {$101010$};\foreach \from/\to in
{2/4}
\draw[->,thick](\from)--(\to);
        \node[Vertex] (5) at (3,-2) {$101100$};
\draw[->,thick](2)--(5);

\path[->, thick] (2) edge[loop left] (2) ;

\node[black!50]  at (0.4,-0.5) {\footnotesize $1$};

\node[black!50]  at (0.4,-1.5) {\footnotesize $\nicefrac{1}{6}$};

\node[black!50]  at (2.1,-1.4) {\footnotesize $\nicefrac{1}{3}$};

\node[black!50]  at (-2.1,-1.4) {\footnotesize $\nicefrac{1}{3}$};

\node[black!50]  at (-1.65, -0.60) {\footnotesize $\nicefrac{1}{6}$};
\end{tikzpicture}
\end{center}
\caption{An example of the move defined in Definition~\ref{def: move}. In the first step we must place the leading card at the bottom of the deck.  In the second step the leading card can be inserted in any of the $6$ positions. The weights on the directed edges indicate the probability of going from one deck to another.  }
\label{fig:moveexample}
\end{figure}

\begin{rmk}
There is an obvious symmetry in the process just described: if the deck $d_1 \cdots d_{2n}$ moves to $d'_1 \cdots d'_{2n}$ with probability $p$, then the deck $(1 - d_1) \cdots (1 - d_{2n})$ (that we get by swapping the roles of the two colors) moves to $(1 - d'_1) \cdots (1 - d'_{2n})$ with probability $p$.  Therefore, going forward, we ``fold'' our Markov chain in half and identify the two decks $d_1 \cdots d_{2n}$ and $(1 - d_1) \cdots (1 - d_{2n})$; in particular, we only consider decks beginning with $0$, swapping colors if necessary.
\label{01symmetry}
\end{rmk}

For instance, the move sends $011001$ to $110010$ which is equivalent to $001101$ by symmetry. Considering Remark~\ref{01symmetry}, we have halved the number of decks we observe for case $n$ to $\frac{1}{2} \binom{2n}{n}$ decks. In Figure~\ref{fig:moveexamplesymmetry}, we redraw Figure~\ref{fig:moveexample} with the new convention.

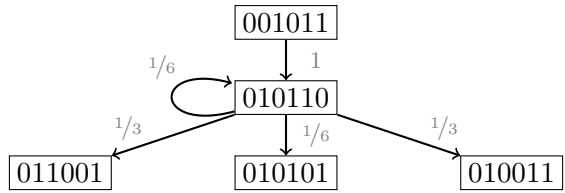
\begin{figure}[H]
\begin{center}
   \begin{tikzpicture}  
   \tikzset{Vertex/.style = {rectangle,inner sep=1mm, draw=black}}  

  \node[Vertex] (1) at (0,0)  {$001011$}; 
  \node[Vertex] (2) at (0,-1) {$010110$};
\foreach \from/\to in
{1/2}
\draw[->,thick](\from)--(\to);
    \node[Vertex] (3) at (-3,-2) {$011001$};
    \foreach \from/\to in
{2/3}
\draw[->,thick](\from)--(\to);
      \node[Vertex] (4) at (0,-2) {$010101$};\foreach \from/\to in
{2/4}
\draw[->,thick](\from)--(\to);
        \node[Vertex] (5) at (3,-2) {$010011$};
\draw[->,thick](2)--(5);

\path[->, thick] (2) edge[loop left] (2) ;

\node[black!50]  at (0.4,-0.5) {\footnotesize $1$};

\node[black!50]  at (0.4,-1.5) {\footnotesize $\nicefrac{1}{6}$};

\node[black!50]  at (2.1,-1.4) {\footnotesize $\nicefrac{1}{3}$};

\node[black!50]  at (-2.1,-1.4) {\footnotesize $\nicefrac{1}{3}$};

\node[black!50]  at (-1.65, -0.60) {\footnotesize $\nicefrac{1}{6}$};
\end{tikzpicture}
\caption{The example in Figure~\ref{fig:moveexample}, relabeled as in Remark~\ref{01symmetry}.}
\label{fig:moveexamplesymmetry}
\end{center}
\end{figure}

Following Remark~\ref{01symmetry}, we may redefine r-decks as the decks that begin and end with a $0$ and d-decks as the decks that begin with a $0$ and end with a $1$.  When $n = 3$, the r-decks are $010110$, $011010$, $001110$ and $011100$,  and the d-decks are  $010101$, $001011$, $001101$, $011001$, $010011$ and $000111$ --- see Figure~\ref{markov chain n=3}.

It will turn out (see Corollary~\ref{prop:absorbing}) that our Markov chain is of a special type, which we define next.

\begin{definition}
An \textbf{absorbing state} in a Markov chain is a state that once entered is impossible to leave.
\label{def:absstate}
\end{definition}
In our case, the absorbing state is the deck of perfectly interleaved $0$s and $1$s. This is because our move forces us to continually place the leading card at the bottom of the deck since a perfectly alternating deck will always remain a d-deck. See the absorbing state, $010101$, for $n=3$ at the top of Figure~\ref{markov chain n=3}.
\begin{definition}
An \textbf{absorbing Markov chain} has the following properties:
\begin{enumerate}[i.]
\item It contains at least one absorbing state.
\item It is possible to go from every other state to an absorbing state (not necessarily in a single step).
\end{enumerate}
When a Markov process reaches an absorbing state, it is said to be \textbf{absorbed}.
\end{definition}

\begin{figure}
\begin{center}
\begin{tikzpicture}
\tikzstyle{edge}=[->]
\node[rectangle, dashed, draw = black!30, ](1) at (0,0){$010101$};
\node[rectangle, fill = blue!20,draw = black!20](9) at (2.5,-1){$011010$};
\node[rectangle, fill = blue!20,draw = black!20](3) at (2.5,-2){$001101$};
\node[rectangle, fill = blue!20,draw = black!20](8) at (2.5,-3){$011001$};
\node[rectangle, fill = blue!20,draw = black!20](5) at (2.5,-4){$010011$};

\node[rectangle, fill = blue!10,draw = black!20](7) at (-2.5,-1){$010110$};
\node[rectangle, fill = blue!10,draw = black!20](2) at (-2.5,-2){$001011$};

\node[rectangle, draw = black!20, fill=red!10](4) at (0,-5){$001110$};
\node[rectangle, draw = black!20,fill=red!10](6) at (0,-6){$000111$};
\node[rectangle, draw = black!20, fill=red!25](10) at (-0,-7){$011100$};

\path[->](9) edge [bend right =20] node[above]{\lbl{$\nicefrac{1}{6}$}} (1)
 (7) edge [bend left=20] node[above]{\lbl{$\nicefrac{1}{6}$}} (1);

\path[->](3) edge node[left]{\lbl{1}} (9);
\path[->](8) edge node[left]{\lbl{1}} (3);
\path[->](5) edge node[left]{\lbl{1}} (8);
\path[->](6) edge node[left]{\lbl{1}} (4);
\path[->](2) edge node[left]{\lbl{1}} (7);

\path[->](1) edge [loop above] node[above]{\lbl{1}} (1);

\path[->](4) edge [loop above] node[above]{\lbl{$\nicefrac{1}{3}$}
}(4)
 (4) edge [out=135,in=315] node[left]{\lbl{$\nicefrac{1}{6}$}} (7)
 (4) edge [out=45,in=225] node[left]{\lbl{$\nicefrac{1}{6}$}} (9)
(4) edge [bend left=60] node[right]{\lbl{$\nicefrac{1}{3}$}} (10)
;

\path[->](9) edge [loop above=45] node[right]{\lbl{$\nicefrac{1}{6}$}} (9)
 (9) edge [bend left] node[right]{\lbl{$\nicefrac{1}{3}$}} (3)
 (9) edge [out=60, in=150, bend right] node[below]{\lbl{$\nicefrac{1}{3}$}} (2);
 
\path[->](7) edge [loop above] node[left]{\lbl{$\nicefrac{1}{6}$}} (7)
 (7) edge [out=-15,in=170] node[above]{\lbl{$\nicefrac{1}{3}$}} (8)
 (7) edge [out=-30,in=180] node[left]{\lbl{$\nicefrac{1}{3}$}} (5);

\path[->](10) edge [loop below] node[below]{\lbl{$\nicefrac{1}{6}$}} (10)
 (10) edge [bend right] node[right]{\lbl{$\nicefrac{1}{6}$}} (5)
(10) edge [bend left] node[left]{\lbl{$\nicefrac{1}{6}$}} (2)
 (10)edge node[left]{\lbl{$\nicefrac{1}{2}$}} (6);
 
\end{tikzpicture}
\end{center}
\caption{Markov chain for $n=3$. Each deck represents a state in a Markov chain and the weights on the directed edges indicate the probability of going from one deck to another. The absorbing state $010101$ is at the top of the figure.}
\label{markov chain n=3}
\end{figure}
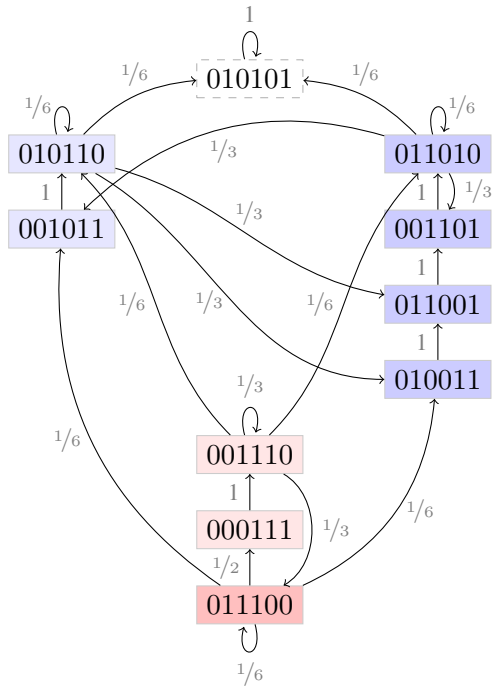

Starting from any state in any absorbing Markov chain, the process is eventually absorbed with probability $1$. Thus, a natural question about the behavior of an absorbing Markov chain is to compute the average number of steps needed to be absorbed, starting from a given state.  There is a standard approach to compute this quantity, which we sketch now (following \cite{FM:1}).

\begin{definition}
The \textbf{transition matrix} of a Markov chain is the square matrix whose $a_{ij}$ entry is the probability of going from state $i$ to state $j$ in one step. 
\end{definition}
By definition, the sum of the row entries of a transition matrix is $1$. See Figure~\ref{transition matrix 2} for the transition matrices associated to our Markov chain in the cases $n = 2$ and~$3$.

Given an absorbing Markov chain, it is always possible to order the states with the absorbing states preceding all others.  In this case, the resulting transition matrix $T_n$ will have the block structure
\[
T_n=
\left[
\begin{array}{c|c}
I & O \\
\hline
R & Q
\end{array}
\right]
\]
where $I$ is an identity matrix whose size is equal to the number of absorbing states and $O$ is a matrix of all $0$s.  For example, the matrices in Figure~\ref{transition matrix 2} are in this form, with a single absorbing state, so that $O$ is a row vector of all $0$s and $R$ is a column vector.

\begin{figure}
\[
\begin{blockarray}{cccc}
\begin{block}{(ccc)c}
  1 & 0 & 0 & 0101\\
\frac{1}{4} & \frac{1}{4} & \frac{1}{2} &0110\\
 0 & 1 &0& 0011 \\
\end{block}
\end{blockarray}
\qquad\quad
\begin{blockarray}{ccccccccccc}
\begin{block}{(cccccccccc)c}
1 &&&&&&&&&  & 010101 \\
\frac{1}{6} & \frac{1}{6} & \frac{1}{3} & & & & \frac{1}{3} &&&& 011010 \\
& 1 &&&&&&&&& 001101 \\
&& 1 &&&&&&&& 011001 \\
&&& 1 &&&&&&& 010011 \\
\frac{1}{6} &&& \frac{1}{3} & \frac{1}{3} & \frac{1}{6} &&&&& 010110 \\
&&&&& 1 &&&&& 001011 \\
& \frac{1}{6} &&&& \frac{1}{6} && \frac{1}{3} && \frac{1}{3}& 001110 \\
&&&&&&& 1 &&& 000111 \\
&&&&\frac{1}{6}&&\frac{1}{6}&&\frac{1}{2}&\frac{1}{6}& 011100 \\
\end{block}
\end{blockarray}
\]
\caption{Transition matrices for $n=2$ (left) and for $n = 3$ (right, with entries equal to $0$ omitted), with rows labeled by the corresponding deck. For instance, the $(2,2)$ entry $\frac{1}{4}$ of the matrix at left is the probability of going from $0110$ to itself in one step (which can only happen if the top card is inserted into the top position), and the $(2,3)$ entry $\frac{1}{2}$ is the probability of going from $0110$ to $0011$ in one step (which happens if the top card is inserted into the third or fourth position in the deck).}
\label{transition matrix 2}
\end{figure}

\begin{definition}
The \textbf{fundamental matrix} $N$ of an absorbing chain is the matrix given by $(I-Q)^{-1}$. 
\end{definition}
\begin{theorem}[{\cite[\S4.8]{FM:1}}]
\label{thm:fundamental matrix}
Adding the row entries of $N$, or equivalently multiplying on the right by the column vector of all $1$s, gives the mean number of steps until absorption from each starting state.
\end{theorem}

Applying Theorem~\ref{thm:fundamental matrix} to the matrix for $n = 2$ (see Figure~\ref{transition matrix 2}), we find that the fundamental matrix for $n=2$ is
\[
N_2=
\left(
I_2 - \begin{pmatrix}\frac{1}{4} & \frac{1}{2} \\ 1 & 0\end{pmatrix}
\right)^{-1}
=
\begin{blockarray}{ccc}
0110 & 0011 & \\[6pt]
\begin{block}{(cc)c}
   4 &2 & 0110\\
   4 & 3& 0011\\
\end{block}
\end{blockarray}
 \]
and so it takes on average $6$ steps to be absorbed beginning at $0110$ and $7$ steps to be absorbed beginning at $0011$.

The above algorithm makes it possible for us to compute the average number of moves it takes for any deck to reach the absorbing state for a given $n$. However, as $n$ increases, the number of states in the Markov chain increases rapidly -- for a given $n$ we have $\frac{1}{2} \binom{2n}{n}$ decks, and so the matrix to be inverted has dimensions $\left(\frac{1}{2} \binom{2n}{n}  - 1\right)\times \left(\frac{1}{2} \binom{2n}{n}  - 1\right)$ -- which makes exact computation challenging even for modest values of $n$.  Moreover, the exact values (discussed further in Section~\ref{sec:exact values}) do not seem to suggest any nice formulas or structure.  Therefore, in the next section, we turn our attention to broader structures in the Markov chain that will allow us to study its behavior systematically.

\section{Tiers}
\label{3}

In this section, we study a global structure of \emph{tiers} on the Markov chain.  The main result of the section (Theorem~\ref{prop:tiers}) establishes that the process always passes through the tiers in order, proceeding from tier $k$ to tier $k-1$; and we use this to conclude that the process is an absorbing Markov chain, as promised earlier.  We also solve some enumerative questions related to the tiers.

\begin{definition}
A \textbf{block} in a deck is a maximum substring made of consecutive copies of cards of a single color.
\label{def:block}
\end{definition}
For example, for $n = 5$, the deck $0110001110$ is a string of length $2n =10$ with three blocks of $0$s and two blocks of $1$s.

\begin{definition}
We say that a deck belongs to \textbf{tier $k$} if it has exactly $n - k$ blocks of $1$s.
\label{def:tier}
\end{definition}

One could rephrase Definition~\ref{def:tier} in various equivalent ways: since we consider only decks beginning with $0$, each block of $1$s is preceded by a $0$, and hence a deck is in tier $k$ if and only if it has precisely $n - k$ consecutive substrings of the form $01$; further, since each $1$ is preceded by either a $0$ or a $1$, a deck belongs to tier $k$ if and only if it has precisely $k$ consecutive substrings of the form $11$.  Since each deck $D$ contains exactly $n$ copies of $1$, it may have as few as one block of $1$s (if all the $1$s are consecutive in $D$) or as many as $n$ (if the deck alternates $0101\ldots$).  Thus, the tiers are numbered from $0$ to $n-1$.
Our first result shows that the tier structure determines the global behavior of the Markov chain.

\begin{theorem}
\label{prop:tiers}
A single move from a deck in tier $k$ either takes the deck to another deck in tier $k$ or to a deck in tier $k-1$.
\end{theorem}
\begin{proof}
Consider a deck $D$ in tier $k$, and let $D'$ be a deck that can be reached from $D$ in one move.  First suppose that reaching the deck $D'$ does not involve switching the colors $0$ and $1$ (as in Remark~\ref{01symmetry}; this happens when the second card in deck $D$ is color $0$ or when $D$ is an r-deck and the top card is inserted in its same position).  If $D$ is an r-deck, we consider two possibilities for where the leading $0$ could be inserted: either it is inserted next to another $0$, leaving the number of blocks of $1$s as $n-k$, thereby remaining in tier $k$; or, the leading $0$ is inserted between two adjacent $1$s, thus increasing the number of blocks of $1$s from $n-k$ to $n-k+1$, sending the deck to tier $k-1$. For decks that end with a $1$, the move sends the leading $0$ to the bottom of the deck, thereby preserving the number of blocks of $1$s.

If instead the second card in $D$ is color $1$ and the top card is not inserted in the top position, then we must apply the color swap as in Remark~\ref{01symmetry} to produce $D'$.  In this case, after the top card is inserted but before swapping colors, the top card of the deck will be color $1$ and the bottom card in the deck will be color $0$ (either because $D$ was an r-deck or because the top card of $D$ was moved to the bottom).  Since block colors alternate between $1$ and $0$, it follows that there must be the same number of $0$-blocks and $1$-blocks in $D'$, and, as in the preceding paragraph, this number must be $n - k$ or $n - k + 1$.  Thus in this case we also see that $D'$ belongs to tier $k$ or $k - 1$.
\end{proof}
Theorem~\ref{prop:tiers} is illustrated for the case $n = 3$ in Figure~\ref{fig:structure} (a redrawing of Figure~\ref{markov chain n=3}). For example, in Figure~\ref{fig:structure}, no arrows point from a deck in tier $2$ to a deck in tier $0$, nor from a deck in tier $1$ to a deck in tier $2$: we can only descend down tiers, without skipping any.


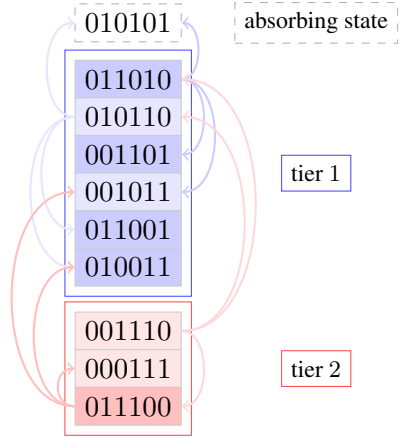
\begin{figure}
\begin{center}
\begin{tikzpicture}
\node[rectangle, dashed, draw = black!30, ](1) at (0,-.25){$010101$};
\node[rectangle, draw = black!20,fill = blue!20](9) at (0,-1){$011010$};
\node[rectangle, draw = black!20,fill = blue!10](7) at (0,-1.5){$010110$};
\node[rectangle, draw = black!20,fill = blue!20](3) at (0,-2){$001101$};
\node[rectangle, draw = black!20,fill = blue!10](2) at (0,-2.5){$001011$};
\node[rectangle,  draw = black!20,fill = blue!20](8) at (0,-3){$011001$};
\node[rectangle, draw = black!20,fill = blue!20](5) at (0,-3.5){$010011$};
\node[rectangle, draw = black!20,fill = red!10](4) at (0,-4.35){$001110$};
\node[rectangle, draw = black!20,fill = red!10](6) at (0,-4.85){$000111$};
\node[rectangle, draw = black!20,fill =red!25](10) at (0,-5.35){$011100$};
\node[draw,red!70,  fit=(4) (6) (10), scale=1] {};
\node[draw, blue!70,  fit=(7) (2) (9) (3) (8) (5), scale=1] {};
\node[rectangle, draw=blue!70](t1) at (2.5,-2.25){\footnotesize{tier $1$}};
\node[rectangle, draw=red!70](t2) at (2.5,-4.85){\footnotesize{tier $2$}};
\node[rectangle, dashed, draw=black!30](abs) at (2.5,-0.25){\footnotesize{absorbing state}};
\path[->](9) edge [thick,blue!20,out=0, in=0] (2)
edge [thick,blue!20,out=0, in=0] (3)
edge [thick,blue!20,out=0, in=0] (1);
\path[->](7)edge [thick,blue!10,out=180, in=180] (8)
edge [thick,blue!10,out=180, in=180] (5)
edge [thick,blue!10,out=180, in=180] (1);

\path[->](4) edge [thick,red!15,out=0, in=0] (7)
edge [thick,red!15,out=0, in=0] (9)
edge [thick,red!15,out=0, in=0] (10);
\path[->](10)edge [thick,red!25,out=180, in=180, looseness=1.5] (6)
edge [thick,red!25,out=180, in=180] (5)
edge [thick,red!25,out=180, in=180] (2);
\end{tikzpicture}
\end{center}
\caption{Tier structure for $n=3$ based on distance from the absorbing state (tier $0$). The arrows indicate where the r-decks go to besides themselves.}
\label{fig:structure}
\end{figure}

\begin{corollary}
\label{prop:absorbing}
The process created by the move on a randomly shuffled deck of length $2n$ is an absorbing Markov chain, and the unique absorbing state is the perfectly alternating deck of $0$s and $1$s: $0101\cdots 01$.  From any starting deck, the process reaches this absorbing state with probability $1$.
\end{corollary}
\begin{proof}
We know that the deck of alternating $0$s and $1$s is an absorbing state since the move always sends this deck back to itself. It follows from the proof of Theorem~\ref{prop:tiers} that there is a path from every deck in tier $k$ to tier $k - 1$, and therefore eventually to the absorbing state in tier $0$. This deck is unique because if another absorbing state existed, then it would exist in tier $i$ where $i>0$ since there is exactly one deck in tier $0$. However, decks in tiers greater than tier $0$ will eventually descend to tier $0$ by Theorem~\ref{prop:tiers}, and hence they are not absorbing states. Thus, our absorbing state is unique and by \cite[\S4.8]{FM:1}, we know that the probability of any deck in an absorbing Markov chain to be absorbed is $1$.
\end{proof}

We end this section by considering some enumerative properties of the tiers.

\begin{definition}
A \textbf{composition} of an integer $n$ is a sequence of positive integers whose sum is $n$; the individual entries of a composition are its \textbf{parts}.
\label{compositions}
\end{definition}

\begin{prop}
The size of tier $k$ is $\binom{n-1}{k}\binom{n}{k}$.
\label{prop:tier size}
\end{prop}
\begin{proof}
We know that a deck in tier $k$ is a string that begins with a $0$, followed by a string of $n$ $1$s and $n-1$ $0$s such that there are $n-k$ blocks of $1$s separated by the remaining $n-1$ $0$s. We can count the number of ways $n$ $1$s can be split into $n-k$ blocks of $1$s. 
This can be done by thinking of the sizes of the blocks of $1$s as the parts of an integer composition of $n$ (Definition~\ref{compositions}).  It is a standard result in elementary enumeration that the number of compositions of $n$ into $n-k$ parts is $\binom{n-1}{k}$ (see, for example, \cite[\S1.2]{EC1}).  After counting the number of compositions of $n$ $1$s into non-empty parts we want to place the remaining $n-1$ $0$s between the blocks of $1$s allowing them to occupy the positions to the left or right of any block of $1$s. This can be done in $\binom{n}{k}$ ways: there is a one-to-one correspondence between such weak compositions and compositions of $n-1+2$ into $n-k+1$ positive parts.
\end{proof}

\begin{rmk}
By summing over all values of $k$, we get a combinatorial proof of the binomial coefficient identity
\[
\frac{1}{2}\binom{2n}{n}=\sum_{k=0}^{n-1}\binom{n-1}{k}\binom{n}{k}.
\]
The left-hand side counts the number of ways to arrange $n$ $0$s and $1$s in a string counting only strings beginning with a $0$. The right-hand side is simply adding the sizes of all the tiers for a fixed $n$.
\end{rmk}

\begin{rmk}
For large $n$, it follows from Proposition~\ref{prop:tier size} that almost all decks lie in tiers that are ``close to'' the middle.  More precisely: for $k$ the nearest integer to $\frac{n}{2} \pm \sqrt{n\ln(n)}$, we have (following \cite[\S5.4]{Spencer}) that $\binom{n}{k}$ and $\binom{n}{k - 1}$ are both asymptotic to $\binom{n}{\lfloor n/2\rfloor} n^{-2}$ (i.e., the limit of the ratios goes to $1$).  Thus, for large $n$, the total number of decks in any tier whose index differ from $\frac{n}{2}$ by more than $\sqrt{n \ln(n)}$ is at most $n^{-4}$ times the size of tier $\lfloor n/2 \rfloor$, and so the total number of decks in all such tiers is at most $n^{-3}$ times the size of the largest tier.  Consequently all but an infinitesimal fraction of decks lie in the middle $2\sqrt{n\ln(n)}$ tiers.
\end{rmk}

We may also consider how many decks in each tier are r-decks and how many are d-decks.

\begin{prop}
The number of r-decks in tier $k$ is $\binom{n-1}{k}\binom{n-1}{k-1}$ and the number of d-decks in tier $k$ is $\binom{n-1}{k}^2$.
\label{numberofrdecks}
\end{prop}
\begin{proof}
An r-deck has $0$ as the first and last digit of the string. So to get the total number of r-decks we can count the number of ways we can split $n$ $1$s into $n-k$ blocks of $1$s in $\binom{n-1}{k}$ ways, as we did in Proposition~\ref{prop:tier size}, and then we arrange the remaining $n-2$ $0$s between or adjacent to the blocks of $1$s in $\binom{n-1}{k-1}$ ways. Thus, the total number of r-decks in tier $k$ is $\binom{n-1}{k}\binom{n-1}{k-1}$. By Proposition~\ref{prop:tier size} we get the number of d-decks to be $\binom{n-1}{k}\binom{n}{k}-\binom{n-1}{k}\binom{n-1}{k-1}=\binom{n-1}{k}^2$.
\end{proof}

It follows immediately that the ratio of r-decks to d-decks is $\frac{\binom{n-1}{k}\binom{n-1}{k-1}}{\binom{n-1}{k}^2} = \frac{k}{n - k}$, so there are relatively few r-decks in small tiers and many in large tiers.

\section{A complete analysis of tier 1}
\label{sec:tier1}

By Proposition~\ref{prop:tier size}, there are $n(n - 1)$ decks in tier $1$.  In this section, we give exact formulas for the absorption time of all these decks.

\begin{definition}
A \textbf{chain} is a maximal sequence $D_1, D_2, \ldots, D_t$ of decks such that $D_i$ moves deterministically to $D_{i + 1}$ with probability $1$ for $i = 1, \ldots, t - 1$.  The \textbf{length} of a chain is the number of decks it contains.

As a special case, we do not consider the absorbing state (the alternating $01$ deck) to be part of a chain.
\label{def:chain}
\end{definition}

For example, we see in Figure~\ref{markov chain n=3} that when $n = 3$, there are three chains of length longer than $1$: $001011 \to 010110$ and $010011 \to 011001 \to 001101 \to 011010$ in tier $1$ and $000111 \to 001110$ in tier $2$.

By definition, every d-deck belongs to a chain of length at least $2$.  In any chain of length $t$, the decks $D_1$, \ldots, $D_{t - 1}$ are (by definition) d-decks, while the deck $D_t$ must be an r-deck (or else the sequence would not be maximal).  We say that the r-deck is the \textbf{end} of its chain, or that the chain \textbf{ends} in the r-deck $D_t$. 

It will be convenient in what follows to have a notation for the expected absorption time of a deck.

\begin{definition}
Let $\lambda_D$ denote the average number of moves it takes to go from deck $D$ to the absorbing state.
\end{definition}

\begin{rmk} 
\label{lambda+x}
In a chain $D_1 \to D_2 \to \cdots \to D_{t - 1} \to R$ of length $t$, the average number of moves to reach the absorbing state from a d-deck $D_i$ is exactly $\lambda_R$ plus the number of deterministic steps from $D_i$ to $R$:
\[
\lambda_{D_i} = \lambda_R + t - i.
\]
\end{rmk}

\subsection{Naming of Decks in Tier 1}

By Proposition~\ref{numberofrdecks}, there are $n - 1$ chains in tier~$1$.  Indeed, the r-decks in tier $1$ are precisely the decks
\[
0110101\cdots010, \quad 0101101\cdots010, \quad \ldots, \quad 010101\cdots0110
\]
that begin and end with $0$, have a single consecutive pair $11$, and otherwise alternate between $0$s and $1$s.  We denote these decks respectively by $[1, 0]$, $[2, 0]$, \ldots, $[n - 1, 0]$,
so that for $1 \leq m \leq n - 1$ we have
\begin{equation}
\label{def:r-decks in tier 1}
[m, 0] := \underbrace{01 \cdots 01}_{m \text{ copies}} \underbrace{10 \cdots 10}_{n - m \text{ copies}}.
\end{equation}
Every other deck in tier $1$ begins with $0$, ends with $1$, and contains one consecutive copy of $00$ and one consecutive copy of $11$ (otherwise alternating between $0$s and $1$s).  We denote these decks by $[m, k]$ where $1 \leq k \leq 2n - 2m - 1$, as follows:
\begin{equation}
\label{def:d-decks in tier 1}
[m, k] := \begin{cases} \overbrace{01\cdots01}^{(k-1)/2 \text{ copies}} 00 \overbrace{10 \cdots 10}^{m - 1 \text{ copies}} 11 \overbrace{01 \cdots 01}^{n - m - (k + 1)/2 \text{ copies}}, & k \text{ odd}, \\[6pt]
 \underbrace{01\cdots01}_{k/2 - 1 \text{ copies}} 0 \, 11 \underbrace{01 \cdots 01}_{m - 1 \text{ copies}} 00 \, 1 \underbrace{01 \cdots 01}_{n -m - k/2 - 1 \text{ copies}}, & k \text{ even}.
 \end{cases}
\end{equation}

\begin{prop}\label{prop:numberofchainsintier1}
Every deck in tier $1$ is of the form $[m, k]$ for exactly one choice of $m$, $k$ such that $1 \leq m \leq n - 1$ and $0 \leq k \leq 2n - 2m - 1$.  When $k > 0$, the deck $[m, k]$ moves with probability $1$ to the deck $[m, k - 1]$.  Furthermore, the $n - 1$ chains in tier $1$ have lengths $2$, $4$, \ldots, $2n - 2$, with the chain ending at r-deck $[m, 0]$ having length $2n - 2m$.
\end{prop}
\begin{proof}
By comparing \eqref{def:r-decks in tier 1} and \eqref{def:d-decks in tier 1}, we see that they describe a total of $2 + 4 + \ldots + (2n - 2) = n(n - 1)$ distinct decks, which (by Definition \ref{def:tier}) all belong to tier $1$.  Since (by Proposition \ref{prop:tier size}) there are exactly $n(n - 1)$ decks in tier $1$, these are all of the decks.  The fact that $[m, k]$ moves with probability $1$ to $[m, k - 1]$ when $k > 1$ follows from Definition~\ref{def: move} after observing that removing the first $0$ from a deck in \eqref{def:d-decks in tier 1}, moving it to the end, and swapping $0$s and $1$s if necessary gives the appropriate deck in \eqref{def:r-decks in tier 1} or \eqref{def:d-decks in tier 1}.  Finally, the chain structure follows immediately from the previous considerations.
\end{proof}

\subsection{Number of Steps in Tier 1}
\label{subsec:number of steps}

In this section, we compute the expected number of steps to absorption for every deck in tier $1$.

\begin{theorem}
The average number of moves it takes for the r-deck $[m, 0]$ in tier $1$ to reach the absorbing state is given by
$$
\lambda_{[m, 0]} = n^2 + \frac{2n}{3} + m - \frac{n^2 + 2m^2 - m}{4n^2 - 1}
$$
\label{thm:stepsrdeckst1}
\end{theorem}
\begin{proof}
We consider the moves that are possible beginning with the r-deck $[m, 0]$, based on where we insert the leading $0$.  With probability $\frac{1}{2n}$, we insert it back at the front of the deck and end up back at $[m, 0]$, and with probability $\frac{1}{2n}$ we insert it between the two consecutive $1$s and end up at the absorbing state.  Otherwise, we insert the top card next to one of the other $0$s, resulting in $n - 1$ possible decks each occurring with probability $\frac{1}{n}$.  We consider two cases.

If we insert the top $0$ to the right of the $11$ then, after swapping $0$s and $1$s, the $00$ precedes the $11$ in the result, so we are in the $k$ odd case of \eqref{def:d-decks in tier 1}.  In particular, in all such cases, the number of $01$ pairs before the $00$ is $m - 1$, so if the new deck is $[m', k]$ then we have $\frac{k - 1}{2} = m - 1$, or $k = 2m - 1$.  The associated $m'$ values range from $1$ (if we insert the top $0$ immediately following the $11$) to $n - m$ (if we insert the top $0$ at the very end).  So in this case we end up at the decks $[1, 2m - 1]$, $[2, 2m - 1]$, \ldots, $[n - m, 2m - 1]$.  

On the other hand, suppose we insert the top $0$ to the left of the $11$, producing a deck $[m', k]$.  After swapping $0$s and $1$s, the $11$ precedes the $00$ in the result $[m', k]$, so we are in the $k$ even case of \eqref{def:d-decks in tier 1}.  Since the top card was inserted to the left of what was the $11$, the number of cards before $11$ in $[m, 0]$ is the same as the number of cards before $00$ in $[m', k]$.  In particular, the deck $[m, 0]$ has $2m - 1$ cards preceding $11$, whereas $[m', k]$ has $2(k/2 - 1) + 3 + 2(m' - 1)$ cards preceding $00$. Setting these two numbers equal gives $k = 2m - 2m'$.  Therefore, the decks produced in this case are $[1, 2m - 2]$, $[2, 2m - 4]$, \ldots, $[m - 1, 2]$.

Combining the preceding arguments, we have
\[
\lambda_{[m, 0]} = \underbrace{1}_{\text{took a step}} + \underbrace{\frac{1}{2n} \cdot \lambda_{[m, 0]}}_{\text{insert at top}} + \underbrace{\frac{1}{2n} \cdot 0}_{\text{absorbed}} + 
\frac{1}{n}\Bigg[\sum_{i=1}^{m-1}\lambda_{[i, 2m-2i]} + \sum_{j=1}^{n-m} \lambda_{[j, 2m - 1]} \Bigg].
\]
By Proposition~\ref{prop:numberofchainsintier1} and Remark~\ref{lambda+x}, we have $\lambda_{[a, b]} = \lambda_{[a, 0]} + b$, so we can rewrite the preceding equation as
\begin{align*}
\lambda_{[m, 0]} & = 1 + \frac{1}{2n} \cdot \lambda_{[m, 0]} + 
\frac{1}{n}\Bigg[m(m - 1) + \sum_{i=1}^{m-1}\lambda_{[i, 0]} + (2m - 1)(n - m) + \sum_{j=1}^{n-m} \lambda_{[j, 0]} \Bigg] \\
& = \frac{m(2n - m)}{n} + \frac{1}{2n} \cdot \lambda_{[m, 0]} + 
\frac{1}{n}\Bigg[ \sum_{i=1}^{m-1}\lambda_{[i, 0]} + \sum_{j=1}^{n-m} \lambda_{[j, 0]} \Bigg].
\end{align*}
Considering all values of $m$, this is a diagonally dominant system of $n - 1$ affine equations in the $n - 1$ variables $\lambda_{[m, 0]}$. Therefore the solution is unique, and so it suffices to check that the proposed values actually are a solution.  Considering just the sums on the right-hand side, we have
\begin{multline*}
\sum_{i=1}^{m-1}\left(n^2 + \frac{2n}{3} + i - \frac{n^2 + 2i^2 - i}{4n^2 - 1}\right) + \sum_{j=1}^{n-m} \left(n^2 + \frac{2n}{3} + j - \frac{n^2 + 2j^2 - j}{4n^2 - 1}\right) = 
 {}\\{} =
(n - 1)\left(n^2 + \frac{2n}{3} - \frac{n^2}{4n^2 - 1}\right) + 
{}\\{} + 
\left(\frac{m(m-1)}{2} + \frac{(n - m + 1)(n - m)}{2}\right)\left(1 + \frac{1}{4n^2 - 1}\right) -
{}\\{} - \left(\frac{m(m-1)(2m - 1)}{6} + \frac{(n - m + 1)(n - m)(2n - 2m + 1)}{6}\right)\frac{2}{4n^2 - 1} =
{} \\ {} =
n^3 + \frac{n^2}{6} - mn + \frac{(12m^2 - 6m - 7n - 2)n}{6(2n + 1)}.
\end{multline*}
Combining with the other terms on the right-hand side, this becomes
\begin{multline*}
\frac{m(2n - m)}{n} + \frac{1}{2n} \cdot \left(n^2 + \frac{2n}{3} + m - \frac{n^2 + 2m^2 - m}{4n^2 - 1}\right) +
n^2 + \frac{n}{6} - m + \frac{12m^2 - 6m - 7n - 2}{6(2n + 1)} = 
{}\\{}=
n^2 + \frac{2n}{3} + m - \frac{n^2 + 2m^2 - m}{4n^2 - 1},
\end{multline*}
matching the claimed value on the left-hand side.  This completes the proof.
\end{proof}
We can extend Theorem~\ref{thm:stepsrdeckst1} to any deck in tier $1$ based on Remark~\ref{lambda+x}.

\begin{corollary}
The average number of moves it takes for the arbitrary deck $[m, k]$ in tier $1$ to reach the absorbing state is given by
\[
\lambda_{[m,k]} = 
n^2 + \frac{2n}{3} + m - \frac{n^2 + 2m^2 - m}{4n^2 - 1} + k.
\]
\label{cor:t1steps}
\end{corollary}

\begin{proof}
This is an immediate consequence of Remark~\ref{lambda+x} and Theorem~\ref{thm:stepsrdeckst1}.
\end{proof}

\begin{corollary}
The deck in tier 1 with the smallest average number of moves until absorption is
\begin{equation}
[1, 0] = 011\underbrace{0101\cdots01}_{n-2 ~\text{pairs}}0.
\label{eqn:bestdeckt1}
\end{equation}
This deck takes $ \lambda_{[1, 0]} = n^2+\dfrac{2n}{3}+\dfrac{3}{4}-\dfrac{5}{4\left(4n^2-1\right)}$  moves on average to reach the absorbing state.
The deck with the largest average number of moves is  
\begin{equation}
[1, 2n - 3] =  \underbrace{0101\cdots01}_{n-2 ~\text{pairs}}0011.
\label{eqn:worstdeckt1}
\end{equation}
This deck takes $\lambda_{[1, 2n - 3]}= n^2+\dfrac{8n}{3}-\dfrac{9}{4}-\dfrac{5}{4\left(4n^2-1\right)}$  moves on average to reach the absorbing state.
\label{cor:bestworstt1}
\end{corollary}
\begin{proof}
First we consider the minimum.  By Remark~\ref{lambda+x}, this will be achieved by an r-deck $[m, 0]$ for some $m \in \{1, \ldots, n - 1\}$.  By Theorem~\ref{thm:stepsrdeckst1}, we have
\begin{align*}
\lambda_{[m + 1, 0]} - \lambda_{[m, 0]} & =  1 + \frac{2m^2 - m - 2(m + 1)^2 + (m + 1)}{4n^2 - 1} \\
& = \frac{4n^2 - 4m - 2}{4n^2 - 1} \\
& \geq \frac{4n^2 - 4n + 2}{4n^2 - 1} \\
& > 0
\end{align*}
(where in the penultimate step we use $m \leq n - 1$ and in the last step we use $4n^2 - 4n + 1 = (2n - 1)^2 \geq 0$).  Therefore the sequence $(\lambda_{[m, 0]})$ is increasing, and so the minimum value is indeed achieved when $m = 1$.

Next we consider the maximum.  By Remark~\ref{lambda+x}, this will be achieved by a deck $[m, k]$ for which $k = 2n - 2m - 1$, i.e., a deck that is as far as possible from the r-deck in its chain.  Comparing two such decks, we have by Corollary~\ref{cor:t1steps} that
\begin{align*}
\lambda_{[m + 1, 2n - 2(m + 1) - 1]} - \lambda_{[m, 2n - 2m - 1]} & = 1+ \frac{2m^2 - m - 2(m + 1)^2 + (m + 1)}{4n^2 - 1} - 2 \\
& = -1 - \frac{4m - 1}{4n^2 - 1} \\
& < 0.
\end{align*}
Therefore the sequence $(\lambda_{[m, 2n - 2m - 1]})$ is decreasing, and so the maximum value is indeed achieved when $m = 1$.
\end{proof}

\section{Average number of moves for tier 2 and beyond}
\label{sec:tier2+}
In this section, we consider the problem of computing the average number of moves required for decks in tiers further from the absorbing state.  Unfortunately, it seems unlikely that simple closed-form formulas like those in Corollary~\ref{cor:t1steps} exist for higher tiers -- see Section~\ref{sec:exact values} for some evidence along these lines.  Therefore, we instead focus our efforts on bounding the values $\lambda_D$ for $D$ a deck in tier $k$.

We begin with some basic properties of tiers.  First, we compute the length of the longest chain in tier $k$.  

\begin{prop} 
The length of the longest chain in tier $k$ is $2n - 2k$, and there is exactly one chain of this length.
\label{prop:longest chain}
\end{prop}
\begin{proof}
Let us start with the following deck:
\begin{equation}
\label{eq:longest chain}
\underbrace{01010\cdots1010}_{2n-2k-1}\underbrace{00\cdots0}_{k}\underbrace{11\cdots1}_{k+1}
\end{equation}
This deck is in tier $k$, since it begins with $n-k-1$ blocks of $1$s in the leftmost string of $2n-2k-1$ $0$s and $1$s and another block of $1$s towards the rightmost part of the string. Notice how the leftmost string of alternating $0$s and $1$s determine the number of deterministic moves till we reach the following deck after $2n-2k-1$ moves:
\[
\underbrace{00\cdots0}_{k}\underbrace{11\cdots1}_{k+1}\underbrace{0101\cdots010}_{2n-2k-1}
\]
This gives us a chain of length $2n-2k$. This is the longest chain in tier $k$. Say we had a longer chain. That would mean that the starting deck must begin with a string of $2n-2k$ pairs of $01$s followed by a $1$ and must end with a $1$. But this would mean that there are at least $n-k+1$ blocks of $1$s which by definition of a tier, is a deck in tier $k-1$ or below. Thus, there is no chain of length longer than $2n-2k$ in tier $k$. Moreover, the only way to produce a chain of exactly the length $2n-2k$ in tier $k$ is to have $2n-2k-1$ consecutive deterministic steps, and by the same reasoning we are forced to start with the deck \eqref{eq:longest chain}.
\end{proof}

We know by Proposition~\ref{prop:tiers} that a deck in tier $k$ can either stay in tier $k$ or go to tier $k-1$.  We now refine this result by calculating the associated probability of going from an r-deck in tier $k$ to a deck in tier $k-1$ in one move.
\begin{theorem}
\label{thm:where next?}
The probability that an r-deck in tier $k$ goes to tier $k-1$ in one move is $\dfrac{k}{2n}$.
\end{theorem}
\begin{proof}
As mentioned immediately after Definition~\ref{def:tier}, a deck belongs to tier $k$ if and only if it has $k$ pairs of adjacent $1$s.  Thus, following the proof of Proposition~\ref{prop:tiers}, there are $k$ possible places the leading $0$ of an r-deck can insert itself to go to a deck in a lower tier, among $2n$ possible insertion locations. So the probability that an r-deck in tier $k$ goes to tier $k-1$ in one move is $\dfrac{k}{2n}$.
\end{proof}

We now establish notation for the quantities we seek to bound.

\begin{definition}
Let $m_k$ denote the minimum value of $\lambda_D$ for $D$ in tier $k$, let $M_k$ denote the maximum value of $\lambda_D$ for $D$ in tier $k$, and let $M_k^\mathrm{r}$ denote the maximum value of $\lambda_D$ for an r-deck $D$ in tier $k$.
\label{notation}
\end{definition}
We don't need a separate notation for the minimum value of $\lambda_D$ for r-decks in tier $k$ because, by Remark~\ref{lambda+x}, $m_k$ is always achieved on an r-deck.

Our next result gives a recursive bound on the quantity $m_k$.

\begin{prop}
For all $k$ ranging from $1$ to $n-1$, we have
\[
m_k \geq 1 +\frac{k}{2n}m_{k-1}+ \frac{2n-k}{2n}m_k.
\]
\label{min_k}
\end{prop}
\begin{proof}
Let $D$ be a deck in tier $k$ such that $\lambda_D = m_k$.  As just mentioned above, we know that $D$ is an r-deck. Therefore, by Theorem~\ref{thm:where next?}, taking one move from deck $D$ results in a deck in tier $k-1$ with probability $\dfrac{k}{2n}$.  In this case, the expected number of moves remaining is at least $m_{k - 1}$.  The other $1-\dfrac{k}{2n} = \dfrac{2n-k}{2n}$ of the time, after taking one move the deck remains in tier $k$.  In this case, the expected number of moves remaining is at least $m_k$.  Therefore, the expected number of moves \emph{after} the first (random) move is at least $\frac{k}{2n}m_{k-1}+ \frac{2n-k}{2n}m_k$, and the claimed result follows immediately.
\end{proof}

Similarly, we can give a recursive bound on the quantity $M_k$.

\begin{prop}
For $k$ ranging from $1$ to $n-1$, we have
\[
M_k \leq 2n-2k +\dfrac{k}{2n}M_{k-1}+ \dfrac{2n-k}{2n}M_k.
\]
\label{prop:max_k}
\end{prop}
\begin{proof}
Let $D$ be a deck in tier $k$ such that $\lambda_D = M_k$.  By Proposition~\ref{prop:longest chain}, from $D$ it is possible to make at most $2n - 2k - 1$ moves before arriving at an r-deck.  From the r-deck at the end of $D$'s chain, we must make one additional (random) move.  By Theorem~\ref{thm:where next?}, following this move we can either go to tier $k-1$ or remain in tier $k$, with probabilities $\dfrac{k}{2n}$ and $\dfrac{2n-k}{2n}$, respectively. If we end up in tier $k- 1$, then the expected number of remaining moves is at most $M_{k - 1}$, while if we end up in tier $k$, it is at most $M_k$.  The result follows immediately.
\end{proof}

We may use the recurrences just derived to give upper and lower bounds on the expected number of moves to absorption for decks in tier $k$.

\begin{theorem}
For decks in tier $k$ where $1\leq k\leq n-1$, we have the lower bound
\begin{equation*}
m_k \geq n^2+2nH_k -\frac{4}{3}n +\frac{3}{4}-\frac{5}{4(4n^2-1)}
\end{equation*}
and the upper bound
\begin{equation*}
M_k \leq 4n^2H_k-3n^2 -4nk+\frac{20}{3}n -\frac{9}{4}-\frac{5}{4(4n^2-1)},
\end{equation*}
where $H_k=\sum_{i=1}^k\frac{1}{i}$ is the $k$-th harmonic number.
\label{maxmink}
\end{theorem}
\begin{proof}
From the recursive formula for $m_k$ in Proposition~\ref{min_k}, we get 
\[
m_k\Bigg(\frac{2n-(2n-k)}{2n}\Bigg) \geq 1 + \frac{k}{2n}m_{k-1}.
\]
Dividing both sides by $\frac{2n-(2n-k)}{2n}$ and iterating, we have
\begin{align*}
m_k &\geq \frac{2n}{k} + m_{k-1}\\
 &\geq \frac{2n}{k} + \frac{2n}{k-1} + m_{k-2}\\
&\vdotswithin{\geq} \\
 &\geq 2n\left(\sum_{i=0}^{k-2} \frac{1}{k-i}\right) + m_1\\
 &=  2n\left(\sum_{i=0}^{k-1} \frac{1}{k-i}\right) -2n +n^2+\frac{2n}{3}+\frac{3}{4}-\frac{5}{4(4n^2-1)}\\
&=n^2+ 2nH_k-\frac{4n}{3}+\frac{3}{4}-\frac{5}{4(4n^2-1)},
\end{align*}
where in the penultimate step we use Corollary~\ref{cor:bestworstt1}.
Similarly, from the recursive formula for $M_k$ in Proposition~\ref{prop:max_k}, we get
\[
M_k\Bigg(\frac{2n-(2n-k)}{2n}\Bigg) \leq 2n-2k + \frac{k}{2n}M_{k-1}.
\]
Dividing and iterating gives
\begin{align*}
M_k &\leq \frac{4n^2}{k} -4n + M_{k-1}\\
 &\leq \frac{4n^2}{k}-4n + \frac{4n^2}{k-1}-4n  + M_{k-2}\\
 &\vdotswithin{\leq} \\
 &\leq 4n^2\left(\sum_{i=0}^{k-2} \frac{1}{k-i}\right) -4n(k-1) + M_1\\
 &= 4n^2H_k-3n^2 -4nk+\frac{20}{3}n -\frac{9}{4}-\frac{5}{4(4n^2-1)},
\end{align*}
as claimed.
\end{proof}

\section{An improved upper bound}
\label{sec:improved max}

In this section, we refine Proposition~\ref{prop:max_k} in order to give an improved upper bound on the quantity $M_k$, the maximum expected number of moves to absorption for a deck in tier $k$.

\begin{theorem} 
Fix a d-deck $D$ in tier $k$ with $k \geq 2$.  Let $R$ be the r-deck at the end of the chain containing $D$, and suppose that it takes $d$ deterministic steps to reach $R$ from $D$.  Suppose further that $D'$ is a deck such that $R$ moves to $D'$ with positive probability, and that it takes $d'$ deterministic steps to reach an r-deck from $D'$.  Then
\[d+d'\leq \begin{cases}
2n-2k-1 & \text{if } D' \text{ is in tier $k$} \\
2n-2k+2 & \text{if } D' \text{ is in tier $k-1$}.
\end{cases}
\]

\label{d+d'}
\end{theorem}

\begin{proof}
To prove this, we consider several cases and sub-cases.  
First, suppose that $d$ is odd.  In this case, $R$ has the form $0\boxed{\cdots}1\underbrace{0101\cdots 10}_{d}$, where the first (boxed) ellipsis denotes a string of length $2n - d - 2$, containing $n - \frac{d + 3}{2}$ copies of $0$ and $n - \frac{n + 1}{2}$ copies of $1$.  If the first entry of this string is $0$, then (regardless of what random step we take from $R$), the deck $D'$ begins and ends with $0$.  In other words, in this case $D'$ is an r-deck, so $d' = 0$ and the result follows by Proposition~\ref{prop:longest chain}.  So when $d$ is odd, we may assume that in fact $R = 01\boxed{\cdots}1 \, 0101\cdots 10$.  We now consider two cases, depending on whether the first consecutive repetition in $R$ is $00$ or $11$. 
        \begin{itemize}
            \item Case 1: We have $$R = 01\boxed{0101\cdots0100\cdots }1\underbrace{0101\cdots 10}_{d},$$
            with a repeating alternating string $0101\cdots$ of even length preceding the first $00$.  Let the length of this alternating string be $2j$, so that 
            $$R = \underbrace{010101\cdots01}_{2j}\boxed{00\cdots 1}\underbrace{0101\cdots 10}_{d}.$$
            Then $d' \leq 2j$, with equality when our random move inserts the top card no earlier in the deck than the first repeated $00$.  In the expression for $R$ above, the number of blocks of $1$s in the boxed $\boxed{00\cdots 1}$ string must be  $n-k-\dfrac{d-1}{2}-j$ (because $R$ has $n - k$ blocks of $1$'s, being in tier $k$, and $\frac{d - 1}{2} + j$ of these are accounted for in the unboxed region).  Since the boxed region contains at least one $1$, it follows that 
                \[
                j\leq n-k-\dfrac{d+1}{2}
                \] 
                and therefore that $d' \leq 2n - 2k - d - 1 $. 
            
            \item Case 2(a): We have $$R = 01\boxed{0101\cdots 011 \cdots}1\underbrace{0101\cdots 10}_{d},$$ and the top card is \emph{not} inserted between the first pair of consecutive $11$ cards.  In this case, we can write 
            $$R = \underbrace{010101\cdots0}_{2j+1}\boxed{11\cdots 1}\underbrace{0101\cdots 10}_{d}$$
            for some $j \geq 0$, and we have $d' \leq 2j+1$ (with equality if the random move places the leading $0$ after the first pair of $1$s in $R$).  Now, in the boxed region $\boxed{11\cdots 1}$, the number of blocks of $1$s is $n-k-\dfrac{d-1}{2}-j$ (because $R$ has $n - k$ blocks of $1$s, of which $j + \frac{d - 1}{2}$ are already accounted for). Also, the boxed region contains at least two blocks of $1$s (since $k\geq 2$), and therefore 
                 \[j\leq n-k-\dfrac{d+3}{2}.\] Thus $d' \leq 2n - 2k - d - 2$.
            
            \item Case 2(b):  We have $$R = 01\boxed{0101\cdots 011 \cdots}1\underbrace{0101\cdots 10}_{d},$$ and the top card \emph{is} inserted between the first pair of consecutive $11$ cards.  In this case, we can write 
            $$R=\underbrace{010101\cdots0}_{2j+1}\boxed{11 \cdots1}\underbrace{0101\cdots 10}_{d}$$ for some $j \geq 0$.  We denote by $m$ the length of the maximal $0101\cdots$ alternating string that begins immediately following the first consecutive $11$ pair, so that
            $$R = \begin{cases} \underbrace{010101\cdots0}_{2j+1}\boxed{11\boxed{\underbrace{0101\cdots01}_{m}}1\cdots 1}\underbrace{0101\cdots 10}_{d}, & m \text{ even}, \\[18pt]
            \underbrace{010101\cdots0}_{2j+1}\boxed{11\boxed{\underbrace{0101\cdots0}_{m}}0\cdots 1}\underbrace{0101\cdots 10}_{d}, & m \text{ odd}.
            \end{cases}
     $$
            Since in this case we place the leading $0$ between the first consecutive $11$ pair, after the random move we reach the d-deck
            \begin{equation}\label{eq:D' for d odd case 3}
            D' = \underbrace{010101\cdots1}_{2j}010\underbrace{1010\cdots a}_{m}\boxed{a\cdots0}\underbrace{1010\cdots 01}_{d}
            \end{equation}
            where $a \in \{0, 1\}$, and we have $d'=2j+3+m$ deterministic steps to make before reaching another r-deck.  Observe that if $m$ is even, the boxed string in \eqref{eq:D' for d odd case 3} must contain at least one copy of $1$ (or else the deck has more $0$s than $1$s).  The deck $D'$ is in tier $\kappa = k-1$, so there are $n-\kappa$ blocks of $1$s.  Comparing this number to the number of blocks of $1$s forced in \eqref{eq:D' for d odd case 3}, we have the inequality
            \[
      n-\kappa\geq \begin{cases}
              j+1+\frac{m}{2}+ 1 +\frac{d+1}{2}, & m \text{ even} \\
              j+1+\frac{m + 1}{2}+\frac{d+1}{2}, & m \text{ odd}.
              \end{cases}
            \]
            Multiplying through by $2$, this gives
      \begin{align*}
          2n-2\kappa &\geq 2j+ m+ d + 4 \\
            & = d + d' + 1.
      \end{align*}
\end{itemize}

In these cases, we get $d + d' \leq 2n - 2k - 1$ every time we end up in tier $k$, and in the worst case (case 2(b)) $d + d' \leq 2n - 2\kappa - 1 = 2n - 2k + 1$ when we end up in tier $k - 1$.
This completes the argument in the case $d$ is odd.

The case that $d$ is even is very similar.  In this case, if the random deck $R$ begins with $00 \cdots$, then the result is again straightforward by Proposition~\ref{prop:longest chain}.  Otherwise, the random deck $R$ has the form $R = 01\boxed{\cdots}0\underbrace{1010\cdots 10}_{d}$, and we again consider cases depending on whether the first repeated consecutive entries in $R$ are $00$ or $11$.
          \begin{itemize}
          \item Case 1: We have $$R = 01\boxed{0101\cdots0100\cdots }0\underbrace{1010\cdots 10}_{d},$$
          with an alternating $0101\cdots$ string of even length preceding the first $00$. Let the length of this string be $2j$, so that
          $$
            R= \underbrace{010101\cdots01}_{2j}\boxed{00\cdots 0}\underbrace{1010\cdots 10}_{d}.
          $$
          Then $d'\leq 2j$. In the boxed string $\boxed{00\cdots 0}$, the number of blocks of $1$s must be $n-k-\dfrac{d}{2}-j$. This quantity must be positive (or else the deck would have more $0$s than $1$s), so we get that $$j \leq n-k-\dfrac{d}{2} - 1$$ and therefore that $d'\leq 2n-2k-d - 2$.
            \item Case 2(a): We have $$R=01\boxed{0101\cdots 01011\cdots \cdots}0\underbrace{1010\cdots 10}_{d},$$
            where the top card is \textit{not} inserted between the first pair of consecutive $1$s. In this case, we can write
            $$R=\underbrace{010101\cdots10}_{2j+1}\boxed{11\cdots 0}\underbrace{1010\cdots 10}_{d}$$
            for some $j\geq 0$, and $d'\leq2j+1$. In the boxed string $\boxed{11\cdots 0}$, the number of blocks of $1$s is $n-k-\dfrac{d}{2}-j$ (because $R$ has $n-k$ blocks of $1$s, of which $j+\dfrac{d}{2}$ are accounted for). Since the boxed string contains at least one block of $1$s, it follows that $$j\leq n-k-\dfrac{d}{2}-1,$$ and thus $d'\leq 2n-2k-d-1.$
            \item Case 2(b): We have
            $$R=01\boxed{0101\cdots011\cdots}1010\cdots10,$$
            and the top card \textit{is} inserted between the first $11$ pair. As in the case when $d$ was odd, we denote by $m$ the length of the maximal $0101\cdots$ alternating string following the consecutive $11$ pair, so that
  \[
  R = \begin{cases} \underbrace{010101\cdots0}_{2j+1}\boxed{11\boxed{\underbrace{0101\cdots01}_{m}}1\cdots 0}\underbrace{1010\cdots 10}_{d}, & m \text{ even}, \\[18pt]
            \underbrace{010101\cdots0}_{2j+1}\boxed{11\boxed{\underbrace{0101\cdots0}_{m}}0\cdots 0}\underbrace{1010\cdots 10}_{d}, & m \text{ odd},
            \end{cases}
  \]
      for some $j\geq 0$, so that $d'= 2j+3+m$. The d-deck reached by the random move has the form 
           \begin{equation} \label{eq:D' for d even case 3}
           D' =  \underbrace{0101\cdots1}_{2j}010\underbrace{1010\cdots a}_{m}\boxed{a\cdots1}\underbrace{0101\cdots 01}_{d}.
           \end{equation}
           The deck $D'$ is in tier $\kappa=k-1$, so there are $n-\kappa$ blocks of $1$s.  Comparing this number to the number of blocks of $1$s forced in \eqref{eq:D' for d even case 3}, we have the inequality
           \[
n - \kappa \geq \begin{cases}
j+1+\frac{m}{2} + 1 + \frac{d}{2}, & m \text{ even}, \\
j+1+\frac{m + 1}{2} + \frac{d}{2}, & m \text{ odd}.
\end{cases}
\]
            Multiplying through by $2$, this gives
      \begin{align*}
        2n-2\kappa &\geq 2j+3 + m + d \\
              & = d + d'
      \end{align*}
in both cases.
       \end{itemize}

In these cases, we get $d + d' \leq 2n - 2k - 1$ every time we end up in tier $k$, and $d + d' \leq 2n - 2\kappa = 2n - 2k + 2$ when we end up in tier $k - 1$. This completes the proof.
\end{proof}

We now show how to use this result to give better bounds on $M_k$.  We begin with tier $2$; according to Theorem~\ref{maxmink}, we have $M_2 \leq 4n^2 H_2 - 3n^2 + O(n) = 3n^2 + O(n)$.  The next result shows that the leading coefficient can be reduced to $2$.

\begin{prop}
\label{prop:M2r}
We have $\displaystyle \Mr_2 \leq 2n^2+\frac{n}{6}+\frac{3}{4}+\frac{6n-17}{4(4n^2-1)}$.
\end{prop}
\begin{proof}
Choose an r-deck $R$ in tier $2$.  After one move, it lands in tier $1$ with probability $\frac{2}{n}$; in this case, by Corollary~\ref{cor:bestworstt1}, it requires on average at most $M_1 = n^2 + \frac{8n}{3} - \frac{9}{4} - \frac{5}{4(4n^2 - 1)}$ moves to reach the absorbing state.  Alternatively, it could end up at some deck $D$ in tier $2$.  From this deck $D$ there will be some number $d = d(D)$ of deterministic steps to reach an r-deck, followed by a random step (resulting in either a deck in tier $1$ or tier $2$), followed by some more deterministic steps.  By considering the worst-case option for this final r-deck, and applying Theorem~\ref{d+d'}, we get that
\begin{align*}
\Mr_2 &\leq 1+\frac{2}{2n} M_1 +\frac{2n-2}{2n}\max_{d, d', d''}\left( 1 + d+\frac{2}{2n}\left(d'+\Mr_1\right)+\frac{2n-2}{2n}\left(d''+\Mr_2\right)\right) \\
&=1+\frac{M_1}{n}+\frac{n-1}{n}\max_{d, d', d''}\left(1 + \frac{1}{n}(d + d')+\frac{n-1}{n}(d + d'') +\frac{\Mr_1}{n}+\frac{n-1}{n}\Mr_2\right) \\
& \leq 1+\frac{M_1}{n} + \frac{n-1}{n}\left( 1 + \frac{1}{n}(2n-2) +\frac{n-1}{n}(2n - 5)+\frac{\Mr_1}{n}+\frac{n-1}{n}\Mr_2\right).
\end{align*}
Collecting the $\Mr_2$ terms, this becomes
\[
\left(1-\left(\frac{n-1}{n}\right)^2\right)\Mr_2 \leq 1+\frac{M_1}{n} + \frac{n-1}{n}\left( 2n - 4 + \frac{3}{n} +\frac{\Mr_1}{n}\right).
\]
Finally, substituting for the values of $M_1$ and $\Mr_1 = \lambda_{[n - 1, 0]}$ (from Theorem~\ref{thm:stepsrdeckst1}), we get
\[
\Mr_2 \leq 2n^2+\frac{n}{6}+\frac{3}{4}+\frac{6n-17}{4(4n^2-1)},
\]
as claimed.
\end{proof}

\begin{theorem}
\label{thm:refined recursion}
For all $k \geq 3$, we have
\[
M_k^\mathrm{r}\leq \frac{8n^3+4n^2-8kn^2+10nk-3k}{4nk-k^2}+\frac{4n-2k+1}{4n-k}M_{k-1}^\mathrm{r}+\frac{k-1}{4n-k}M_{k-2}^\mathrm{r}.
\]
\end{theorem}
\begin{proof}
We take the same approach as in the proof of Proposition~\ref{prop:M2r}.  Fix an r-deck, $R$, in tier $k$, and consider what happens to the deck in the random process. First, we make a random move. With probability $k/2n$ we go to a deck in tier $k-1$, and with probability $(2n-k)/2n$ we go to a deck in tier $k$. If we go to a deck in tier $k-1$ we make some number $d_1$ of deterministic steps before reaching another r-deck, from which we take one random move; with probability $(k - 1)/2n$ we land in tier $k-2$ and take some number $d_1''$ of deterministic steps to arrive at an r-deck, from which at most $M_{k-2}^\mathrm{r}$ moves are required (in expectation).  Alternatively, we go to a deck in tier $k-1$ with probability $(2n-k+1)/2n$, from which we take some number $d_1'$ of deterministic steps to arrive at an r-deck, from which at most $M_{k-1}^\mathrm{r}$ moves are required (in expectation). If instead, from $R$, we go to tier $k$ again, similar reasoning follows for this two-step recurrence, however here we have $d_0$, $d_0'$ and $d_0''$ indicating the deterministic steps taken at different stages. This gives us the following inequality:
\begin{multline*}
M_{k}^{\mathrm{r}} \leq 1 +\frac{2n-k}{2n}\max_{d_0}\left(d_0+1+\frac{2n-k}{2n}\left(d_0'+M_k^\mathrm{r}\right)+\frac{k}{2n}\left(d_0''+M_{k-1}^\mathrm{r}\right)\right)+\\
\frac{k}{2n}\max_{d_1}\left(d_1+1+\frac{2n-k+1}{2n}\left(d_1'+M_{k-1}^\mathrm{r}\right)+\frac{k-1}{2n}\left(d_1''+M_{k-2}^\mathrm{r}\right)\right).
\end{multline*}
By Theorem~\ref{d+d'} we know that $d_0'\leq 2n-2k-1-d_0$, $d_0''\leq 2n-2k+2-d_0$, $d_1'\leq 2n-2k+1-d_1$, and $d_1''\leq 2n-2k+4-d_1$. Making these replacements gives us
%
%
%
\begin{multline*}
M_{k}^{\mathrm{r}} \leq 1 +\frac{2n-k}{2n}\left(1+\frac{2n-k}{2n}\left(2n-2k-1+M_k^\mathrm{r}\right)+\frac{k}{2n}\left(2n-2k+2+M_{k-1}^\mathrm{r}\right)\right)+\\ \frac{k}{2n}\left(1+\frac{2n-k+1}{2n}\left(2n-2k+1+M_{k-1}^\mathrm{r}\right)+\frac{k-1}{2n}\left(2n-2k+4+M_{k-2}^\mathrm{r}\right)\right).
\end{multline*}
After collecting the terms involving $\Mr_k$ on the left and the terms not involving $\Mr_j$ for any $j$ on the right, this becomes
\[
\frac{4nk-k^2}{4n^2} \cdot M_{k}^{\mathrm{r}} \leq \frac{8n^3+4n^2-8kn^2+10nk-3k}{4n^2}+\frac{k(4n-2k+1)}{4n^2}M_{k-1}^\mathrm{r}+\frac{k(k-1)}{4n^2}M_{k-2}^\mathrm{r}.
\]
Dividing through by the prefactor on the left gives the result.
\end{proof}

\begin{theorem}
\label{thm:maxMr}
For any positive integer $k < n$, we have the upper bound 
\[
\Mr_k \leq 2n^2 H_k - n^2 - n(k - 8/3)
\]
where $H_k$ is the $k$-th harmonic number.
\end{theorem}

\begin{proof}
We proceed by induction, using the recurrence in Theorem~\ref{thm:refined recursion}.  We begin by checking the base cases $k = 1, 2$. Using Theorem~\ref{thm:stepsrdeckst1}, it is easy to see for $k = 1$ that
\[M_1^\mathrm{r}=n^2+\frac{5n}{3}-\frac{7}{4}+\frac{5(4n-3)}{4(4n^2-1)}\leq n^2 +\frac{5n}{3}\]
for $n \geq 2$.  Similarly, for $k = 2$, by Proposition~\ref{prop:M2r}, we must check
\[M_2^\mathrm{r} \leq 2n^2+\frac{n}{6}+\frac{3}{4}+\frac{6n-17}{4(4n^2-1)}\leq 2n^2+\frac{2n}{3}\] 
for $n \geq 3$. After algebraic simplification, the inequality becomes
\[\frac{n(2n+1)(n-2)+5}{4n^2-1}\geq 0,\] which is clearly valid for $n\geq 3$.  We now proceed to the inductive step: assuming for some fixed $k \geq 3$ that $M_{k-1}^\mathrm{r}\leq 2n^2H_{k-1}-n^2-n((k-1)-8/3)$ and $M_{k-2}^\mathrm{r}\leq 2n^2H_{k-2}-n^2-n((k-2)-8/3)$, we aim to deduce that $\Mr_k \leq 2n^2 H_k - n^2 - n(k - 8/3)$. By  Theorem~\ref{thm:refined recursion}, we have
\begin{align*}
M_k^\mathrm{r} & \leq \dfrac{8n^3+4n^2-8kn^2+10nk-3k}{4nk-k^2} + \frac{4n-2k+1}{4n-k}M_{k-1}^\mathrm{r}+\frac{k-1}{4n-k}M_{k-2}^\mathrm{r} \\
&\leq  \dfrac{8n^3+4n^2-8kn^2+10nk-3k}{4nk-k^2} + \frac{4n-2k+1}{4n-k}\left(2n^2H_{k-1}-n^2-n((k-1)-8/3)\right) + {} \\& \qquad +\frac{k-1}{4n-k}\left(2n^2H_{k-2}-n^2-n((k-2)-8/3)\right)\\
& =\dfrac{8n^3+4n^2-8kn^2+10nk-3k}{4nk-k^2}+\frac{4n-2k+1}{4n-k}\left(\frac{2n^2}{k-1} + n\right) + \frac{k-1}{4n-k}\left(2n\right) + {} \\& \qquad + 2n^2H_{k-2} - n^2 - n(k - 8/3).
\end{align*}
So it would suffice to show that the above expression is less than or equal to $2n^2H_k-n^2-n(k-8/3)$, or, equivalently, that their difference is positive. After algebraic simplification, the expression of their difference becomes
\begin{multline*}
2n^2\left(\frac{1}{k}+\frac{1}{k-1}\right) - \dfrac{8n^3+4n^2-8kn^2+10nk-3k}{4nk-k^2} - \frac{4n-2k+1}{4n-k}\cdot \frac{2n^2}{k-1} - {}\\ {} - \frac{(4n-2k+1) \cdot n}{4n-k} - \frac{(k-1) \cdot 2n}{4n-k} = {} \\
= \frac{4kn^2 - 9kn - 4n^2 + 3k}{(4n - k)k}.
\end{multline*}
This expression is positive when $3 \leq k < n$ because, for example, it can be written as
\[
\frac{((4n - 1)(n - 2) + 1)(k - 3) + (8n - 3)(n - 3) }{(4n - k)k}. \qedhere
\]
\end{proof}

\begin{corollary}
\label{cor:betterMbound}
We have
\[
M_k \leq 2 n^2 H_k - n^2 - n(k - 14/3)-2k-1.
\]
\end{corollary}
\begin{proof}
This is an immediate consequence of Theorem~\ref{thm:maxMr} and Proposition~\ref{prop:longest chain}.
\end{proof}

\section{Remarks and Open Questions}\label{4}

\subsection{Exact values}
\label{sec:exact values}

\begin{table}
\[
\begin{array}{c|cccc}
m_k & n = 2 & n = 3 & n = 4 & n = 5 
\\\hline\\[-.75em]
 k = 1 & 6 & \frac{82}{7} & \frac{1222}{63} & \frac{2878}{99} 
 \\
 & 0110 & 011010 & 01101010 & 0110101010 
 \\\hline\\[-.75em]
 k = 2 && \frac{95}{7} & \frac{4600037}{176148}  & \frac{710077708867121}{18420324934572} 
 \\
       && 001110 & 00111010 & 0011101010 \\\hline\\[-.75em]
 k = 3 &&& \frac{163571425}{5460588} & \frac{554816010312075512538895}{12684262385736134591112
} \\
               &&& 00011110 & 0001111010 \\\hline\\[-.75em]
 k = 4 &&&& \frac{13189822020736497658538011}{279053772486194961004464}\\
                        &&&& 0000111110
\end{array}
\]

\[
\begin{array}{c|cccc}
\Mr_k & n = 2 & n = 3 & n = 4 & n = 5 \\\hline\\[-.75em]
 k = 1 & 6 & \frac{88}{7} & \frac{1334}{63} & \frac{3148}{99} 
 \\
 & 0110 & 010110 & 01010110 & 0101010110 
 \\\hline\\[-.75em]
 k = 2 && \frac{120}{7} & \frac{561691}{19572} & \frac{786215418814907}{18420324934572
} \\
       && 011100 & 01011100 & 0101011100\\\hline\\[-.75em]
 k = 3 &&& \frac{84904643}{2730294} & \frac{295681831813463606167247}{6342131192868067295556}
\\
               &&& 01111000 & 0101111000\\\hline\\[-.75em]
 k = 4 &&&& \frac{27093124726027530844991687}{558107544972389922008928}
\\
                        &&&& 0011111000
\end{array}
\]

\[
\begin{array}{c|cccc}
M_k & n = 2 & n = 3 & n = 4 & n = 5 \\\hline\\[-.75em]
 k = 1 & 7 & \frac{103}{7} & \frac{1537}{63} & \frac{3571}{99} 
 \\
 & 0011 & 010011 & 01010011 & 0101010011 
 \\\hline\\[-.75em]
 k = 2 && \frac{120}{7} & \frac{5128481}{176148} & \frac{802179333539981}{18420324934572
} \\
       && 000111 & 01000111 & 0101000111\\\hline\\[-.75em]
 k = 3 &&& \frac{169032013}{5460588} & \frac{592868797469283916312231}{12684262385736134591112
} \\
               &&& 00001111 & 0100001111\\\hline\\[-.75em]
 k = 4 &&&& \frac{13468875793222692619542475}{279053772486194961004464}
\\
                        &&&& 0000011111
\end{array}
\]
\caption{The values of $m_k$, $\Mr_k$, and $M_k$ for all $1\leq k < n \leq 5$, and the decks that achieve these values.}
\label{table:exact data}
\end{table}

As discussed in Section~\ref{2}, for any fixed $n$ one can use Theorem~\ref{thm:fundamental matrix} to compute the exact values of $\lambda_D$ for all decks $D$.  In light of Corollary~\ref{cor:t1steps}, one might hope for relatively simple exact formulas.  Unfortunately, the data are not accommodating: in Table~\ref{table:exact data}, we give the exact values of $m_k$, $M_k$, and $\Mr_k$ for $1 \leq k < n \leq 5$.  Even basic features of the numbers $\lambda_D$ are imposing for example, it is not hard to show that the least common denominator of $\{\lambda_D : D \text{ in tier } 1\}$ is $4n^2 - 1 = (2n - 1)(2n + 1)$, but the sequence of common denominators for tier $2$ begins 
\begin{align*}
7 & = 7^1, \\
176148 &= 2^2 \cdot 3^3 \cdot 7 \cdot 233, \\
18420324934572 & = 2^2 \cdot 3^3 \cdot 11 \cdot 15505324019, \\
438067323206466940220363196436798 & = 
2 \cdot 11 \cdot 13^2 \cdot 31 \cdot 1195848422971 \cdot 3178290901266961,
\end{align*}
for $n = 3, 4, 5, 6$.  Both the growth rate and lack of simple factorizations are discouraging; unsurprisingly, the situation appears worse in higher tiers.  There are some apparent patterns to the decks that appear in Table~\ref{table:exact data} (particularly those that achieve $m_k$ and $M_k$), but it is not clear how to prove or take advantage of them.

%
%

\subsection{Better bounds}

If exact values are out of the question, it is natural to ask about the sharpness of the bounds in Theorems~\ref{maxmink} and~\ref{thm:maxMr} and Corollary~\ref{cor:betterMbound}.  These bounds are illustrated for the case $n = 7$ in Figure~\ref{case7}.  (This was the largest value for which we were able to naively apply Theorem~\ref{thm:fundamental matrix} to generate exact values.) 
\begin{figure}
  \begin{center}
    \includegraphics[width=5in]{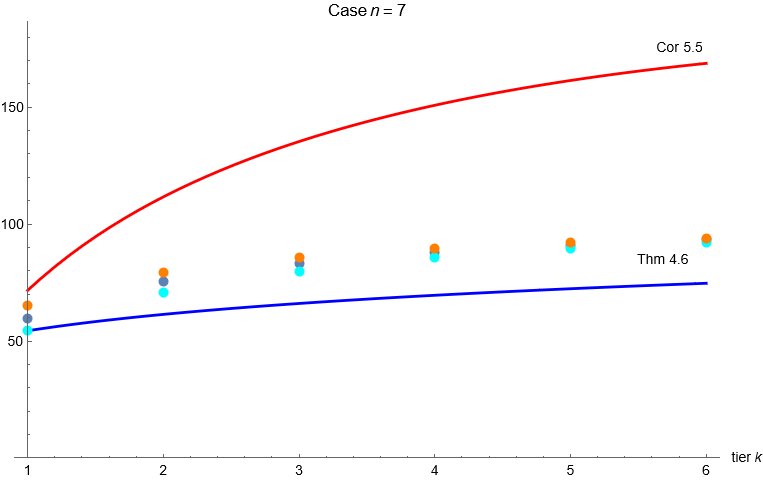}
  \end{center}
  \caption{
    Exact data and bounds for the expected number of steps to absorption when $n=7$, by tier.  The red curve is the upper bound from Corollary~\ref{cor:betterMbound} and the blue curve is the lower bound from Theorem~\ref{maxmink}.  The exact values of $M_k$ and $m_k$ are represented by the orange and cyan dots, respectively; the gray dotes represent the mean of $\lambda_D$ for $D$ in tier $k$.
    }
  \label{case7}
\end{figure}
The figure suggests that there is significant room for improvement, especially in the upper bound.  Further evidence for this thesis may be produced via experimental data for larger $n$.  For example, we used Sage \cite{sagemath} to run $1000$ random trials for every deck in the highest tier $k = n - 1$ with $n = 26$; the results are shown in Table~\ref{table:experiments}.
These data should not be considered overly precise -- in a separate experiment, running ten trials of $1000$ runs each with the deck $0^{26}1^{26}$ produced averages ranging from $1152.186$ to $1217.809$ 
-- but they at least give a sense of the order of magnitude; the largest is $1244.156 \approx 1.84 \times 26^2$.  These experiments (and other similar ones not shown) are compatible with the following conjecture.
\begin{conj}
\label{conj:limit}
We have
\[
\max \{\lambda_D \colon D \text{ any deck}\} \leq C n^2 + o(n^2)
\]
as $n \to \infty$ for some absolute constant $C \leq 2$.
\end{conj}

\begin{table}
\[
\begin{array}{cc}
\text{deck} & \text{average moves $/1000$ trials} \\\hline
0^{26}1^{26} & 1175.637\\
0^{25}1^{26}0 & 1213.973\\
0^{24}1^{26}0^{2} & 1207.898\\
0^{23}1^{26}0^{3} & 1189.077\\
0^{22}1^{26}0^{4} & 1191.580\\
0^{21}1^{26}0^{5} & 1182.054\\
0^{20}1^{26}0^{6} & 1216.356\\
0^{19}1^{26}0^{7} & 1167.585\\
0^{18}1^{26}0^{8} & 1184.297\\
0^{17}1^{26}0^{9} & 1198.983\\
0^{16}1^{26}0^{10} & 1169.055\\
0^{15}1^{26}0^{11} & 1188.905\\
0^{14}1^{26}0^{12} & 1214.038
\end{array}
\qquad
\begin{array}{cc}
\text{deck} & \text{average moves $/1000$ trials} \\\hline
0^{13}1^{26}0^{13} & 1244.156\\
0^{12}1^{26}0^{14} & 1181.733\\
0^{11}1^{26}0^{15} & 1173.931\\
0^{10}1^{26}0^{16} & 1170.929\\
0^{9}1^{26}0^{17} & 1211.762\\
0^{8}1^{26}0^{18} & 1189.327\\
0^{7}1^{26}0^{19} & 1192.200\\
0^{6}1^{26}0^{20} & 1196.258\\
0^{5}1^{26}0^{21} & 1171.347\\
0^{4}1^{26}0^{22} & 1162.706\\
0^{3}1^{26}0^{23} & 1211.565\\
0^{2}1^{26}0^{24} & 1184.151\\
01^{26}0^{25} & 1198.137
\end{array}
\]
\caption{Average number of moves to absorption over $1000$ trials for all decks in the highest tier when $n = 26$.}
\label{table:experiments}
\end{table}

We present some initial thoughts on proving such a conjecture.  The first observation, captured by the next theorem, is that random steps are essentially irrelevant.

\begin{theorem}
Suppose that $D$ is a deck in tier $k$.  The average number of \emph{random} moves before absorption is $2nH_k$, where $H_k$ is the $k$-th harmonic number.
\end{theorem}

\begin{proof}
By Proposition~\ref{thm:where next?}, each time we reach an r-deck in tier $k$, the probability we proceed to tier $k- 1$ is $k/2n$.  Therefore, by a standard result (summing a geometric series), the expected number of random steps necessary to reach tier $k - 1$ from tier $k$ is $2n/k$.  Therefore, by induction, the expected number of random steps to reach tier $0$ from tier $k$ is 
\[
\frac{2n}{1} + \frac{2n}{2} + \frac{2n}{3} + \ldots + \frac{2n}{k} = 2nH_k,
\]
as claimed.
\end{proof}

Since $2nH_k = o(n^2)$, for large $n$ all but an infinitesimal fraction of moves taken starting from an arbitrary deck on the way to tier $0$ must be deterministic.  Our current analysis essentially assumes the worst possible scenario for deterministic steps after each random step. One might hope to get better results by considering more carefully the distribution of possible numbers of deterministic steps.

\subsection{Unbalanced decks}

Consider a version of the game studied above with $m$ red cards and $n$ black cards, where $m < n$.  Some aspects of the preceding analysis continue to hold; in particular, Theorem~\ref{prop:tiers} carries over essentially verbatim: the decks can be divided into $m$ tiers, depending how many blocks of each color there are, and a single random move either returns to the same tier or moves one tier lower.  Can one give bounds on the expected number of steps to tier $0$ (the tier in which all $m$ red cards are in separate blocks)?  In this version of the process, tier $0$ has nontrivial structure, which essentially boil down to a recurrent Markov chain on compositions of $n$ into exactly $m$ parts.  What are the dynamics of this chain?

\subsection{Other variations}

We describe here some variations of our process that might also be interesting to study.
\begin{enumerate}
\item One simple variation would be to consider distributions other than uniform for the random insertion.  In particular, the binomial distribution seems natural.  Some distributions will have large effects on the overall behavior of the process; for example, if the top card is always inserted before position $n$, then the deck $0^{n - 1}1^n0$ will become an additional absorbing state.
\item Another variation would be to replace the top-to-random insertion with another kind of shuffle.  For example, one could consider the process in which the top card is placed on the bottom if they differ in color, and otherwise a random riffle shuffle is applied to the deck.  The deck in which colors alternate would still be an absorbing state, but one should expect a completely different analysis (no tier structure etc.) and possibly a much longer absorption time.  Any other shuffle (random-to-top, block-transposition, \ldots) could be used in place of a riffle shuffle.
\item Placing the top card on the bottom is an extreme example of a \textbf{cut}, where the deck is divided into two intervals and the top interval is placed below the bottom interval.  One could define a new random process in which the deterministic move  is replaced by a random cut.  Cutting an alternating deck always produces an alternating deck, so this chain would still be absorbing.  How does its absorption time compare to ours?  This rule could also be combined with different kinds of shuffles (as in the previous variation).
\item One could consider the reversed version of our rule, in which the top card is moved to the bottom when they \emph{agree} in color, and is inserted randomly when they disagree.  This process has a tendency to collect cards of the same color together -- but unlike our process, it produces a recurrent (rather than absorbing) Markov chain.  What can one say about its steady-state distribution?
\item One could also contemplate decks of more than two colors (or suits).  Are there any interesting variations in this setting?
\end{enumerate}


\bibliography{shuffle-bibliography} 
\bibliographystyle{ieeetr}

\end{document}